\documentclass[12pt,a4paper]{article} 
\usepackage[latin1]{inputenc}
\usepackage{setspace} 
\usepackage{amsmath}
\usepackage{amsfonts}
\usepackage{amssymb}
\usepackage{amsthm}
\usepackage{tikz}
\usepackage{graphicx}
\usepackage{epstopdf}

\author{Anat Amir}
\title{Sharpness of Zapolsky inequality for quasi-states and Poisson brackets}

\newtheorem{thm}{Theorem}

\newtheorem{lem}{Lemma}

\theoremstyle{remark}
\newtheorem{rem}{Remark}
\theoremstyle{definition}
\newtheorem{defi}{Definition}
\newtheorem{exam}{Example}

\begin{document}

\onehalfspacing

\maketitle

\begin{abstract}
Zapolsky inequality gives a lower bound for the $L_1$ norm of the Poisson bracket of a pair of $C^1$ functions on the two-dimensional sphere by means of quasi-states. Here we show that this lower bound is sharp.
\end{abstract}

\section{Introduction and main results}
\subsection{Quasi-states and quasi-measures}
\paragraph*{}
Denote by $C(S^2)$ the Banach algebra of real continuous functions on $S^2$ taken with the supremum norm.
\paragraph*{}
For $F \in C(S^2)$, write $C(F) = \left\{ \varphi \circ F \vert \varphi \in C(Im(F)) \right\}$. That is, $C(F)$ is the closed sub-algebra generated by $F$ and the constant function $1$.
\begin{defi}
A \emph{quasi-state} on $S^2$ is a functional $\zeta : C(S^2) \rightarrow \mathbb{R}$ satisfying: 
\begin{enumerate}
\item $\zeta(F) \geq 0$ for $F \geq 0$.
\item $\forall F \in C(S^2)$, $\zeta$ is linear on $C(F)$.
\item $\zeta(1) = 1$.
\end{enumerate}
\end{defi}
Denote by $\mathcal{Q} (S^2)$ the collection of quasi-states on $S^2$.
\begin{rem} \label{rem:qsp0}
It was proven in \cite{Aa1} that for a quasi-state $\zeta$ and a pair $F,G \in C(S^2)$ we have: 
\[F \leq G \Rightarrow \zeta(F) \leq \zeta(G) \ . \]
\end{rem}
\paragraph*{}
A quasi-state $\zeta$ is \emph{simple} if for every $F \in C(S^2)$, $\zeta$ is multiplicative on $C(F)$.
A quasi-state $\zeta$ is \emph{representable} if it is the limit of a net of convex combinations of simple quasi-states. That is, $\zeta$ is an element of the closed convex hull of the subset of simple quasi-states.
\paragraph*{}
Denote by $\mathcal{C}$ and $\mathcal{O}$ the collections of closed and open subsets of $S^2$ respectively. Write $\mathcal{A} = \mathcal{C} \cup \mathcal{O}$.
\begin{defi}
A \emph{quasi-measure} $\tau$ on $S^2$ is a function $\tau : \mathcal{A} \rightarrow [0, 1]$ satisfying: 
\begin{enumerate}
\item $\tau(S^2) = 1$.
\item For $B_1, B_2 \in \mathcal{A}$ with $B_1 \subset B_2$, $\tau(B_1) \leq \tau(B_2)$.
\item If $\left\{ A_k \right\}_{k=1}^{n} \subset \mathcal{A}$ is a finite collection of pairwise disjoint subsets whose union is in $\mathcal{A}$, then $\tau(\bigcup_{k=1}^{n} A_k) = \sum_{k=1}^{n} \tau (A_k)$.
\item For $U \in \mathcal{O}$, $\tau(U) = \sup \left\{ \tau(K) \ : \ K \in \mathcal{C} \text{ and } K \subset U \right\}$.
\end{enumerate}
\end{defi}
Denote by $\mathcal{M} (S^2)$ the collection of quasi-measures on $S^2$.
A quasi-measure is \emph{simple} if it only takes values of $0$ and $1$.
\paragraph*{}
It was proven in \cite{Aa1} that there exists a bijection between $\mathcal{Q}(S^2)$ and $\mathcal{M}(S^2)$.
For a quasi-state $\zeta$, the corresponding quasi-measure is: 
\[\tau(A) = \left\{ \begin{array}{ll} 
A \in \mathcal{C}, & \inf \left\{ \zeta(F) \ : \ F \in C(S^2) \text{ and } F \geq 1_A \right\} \\ 
A \in \mathcal{O}, & 1 - \tau(S^2 \setminus A) 
\end{array} \right. \ , \]
here $1_A$ is the indicator function on the set $A$.
The corresponding quasi-state to a quasi-measure $\tau$ is defined as follows: 
\[\zeta(F) = \int_{S^2} {F d \tau} = \max_{S^2} {F} - \int_{\min_{S^2} {F}}^{\max_{S^2} {F}} {b_F (x) dx } \ , \]
with $b_F (x) = \tau (\left\{ F < x \right\})$.
It was proven in \cite{Aa2} that this bijection matches simple quasi-states with simple quasi-measures.
For further details about quasi-states and quasi-measures refer to \cite{Aa1} and  for details on simple quasi-states and quasi-measures refer to \cite{Aa2}.
\paragraph*{}
Throughout this paper we will be interested in the extent of non-linearity of a quasi-state. To measure this we will use the following notation: 
\begin{defi}
Let $\zeta$ be a quasi-state and take $F,G \in C(S^2)$. The extent of non-linearity of $\zeta$ can be measured by: 
\[\Pi(F,G) := \vert \zeta(F+G) - \zeta(F) - \zeta(G) \vert \ . \]
\end{defi}
\begin{exam} \label{ex:3point} 
One example of a simple quasi-state is Aarnes' 3-point quasi-state.
\begin{defi} A subset $S \subseteq S^2$ is called \emph{solid} if it is connected and its complement $S^{c} = S^2 \setminus S$ is also connected.
Denote by $\mathcal{C}_s$ the set of all closed and solid subsets of $S^2$ and by $\mathcal{O}_s$ the set of all open and solid subsets of $S^2$. Write $\mathcal{A}_s = \mathcal{C}_s \cup \mathcal{O}_s$.
\end{defi}
Take $p_1, p_2, p_3 \in S^2$ to be three distinct points on the sphere.
Define $\tau : \mathcal{C}_s \rightarrow \left\{ 0,1 \right\}$ by: 
\[\tau(C) = \left\{ \begin{array}{lr} 0, & \#\left\{C \cap \left\{ p_1,p_2,p_3 \right\} \right\} \leq 1 \\ 
1, & \#\left\{ C \cap \left\{ p_1,p_2,p_3 \right\} \right\} \geq 2 \end{array} \right. \ .\]
As proved in \cite{Aa3}, $\tau$ can be extended to a quasi-measure on $S^2$. It is further shown in that article that this extension is in-fact a simple quasi-measure. The simple quasi-state corresponding to the extended quasi-measure is called Aarnes' 3-point quasi-state.
We refer the reader to \cite{Aa3} for the full definition of the extended quasi-measure $\tau$. For our purpose it suffices to note that on $\mathcal{A}_s$, $\tau$ satisfies: 
\[\tau(S) = \left\{ \begin{array}{lr} 0, & \#\left\{S \cap \left\{ p_1,p_2,p_3 \right\} \right\} \leq 1 \\ 
1, & \#\left\{ S \cap \left\{ p_1,p_2,p_3 \right\} \right\} \geq 2 \end{array} \right. \ . \]
\end{exam}
\begin{exam} \label{ex:median} 
Another example of a simple quasi-state is the median of a Morse function.
Let $\Omega$ be an area form on $S^2$.
The \emph{median} of a Morse function $F$ is the unique connected component of a level set of $F$, $m_F$, for which every connected component of $S^2 \setminus m_F$ has area $\leq \frac{1}{2} \cdot \int_{S^2} \Omega$.
Define $\zeta$ on the set of Morse functions as $\zeta(F) = F(m_F)$. As explained in \cite{EPZ}, $\zeta$ can be extended to $C(S^2)$ and is in-fact a quasi-state. For further explanation of the concept of the median and the construction of $\zeta$ we refer the reader to \cite{EPZ}.
It can be easily verified that the quasi-measure corresponding to $\zeta$ is the extension of $\tau : \mathcal{C}_s \rightarrow \left\{ 0,1 \right\}$ defined as: 
\[\tau(C) = \left\{ \begin{array}{lr} 0, & \int_{C} \Omega < \frac{1}{2} \cdot \int_{S^2} \Omega \\ 1, & \int_{C} \Omega \geq \frac{1}{2} \cdot \int_{S^2} \Omega \end{array} \right. \]
to a quasi-measure on $S^2$ as in \cite{Aa3}. In-fact, as explained in \cite{Aa3}, this extension is a simple quasi-measure, and hence $\zeta$ is a simple quasi-state.
\end{exam}
\subsection{Poisson bracket}
\paragraph*{}
Let $\omega$ be an area form on $S^2$.
Given a hamiltonian $F : S^2 \rightarrow \mathbb{R}$, we define the hamiltonian vector field $I dF : S^2 \rightarrow T S^2$ by the formula: 
\[ dF(x) (\eta) = \omega(\eta, I dF (x)) \ , \forall x \in S^2 , \eta \in T_x S^2\ . \]
The hamiltonian flow with hamiltonian function $F$ is the one-parameter group of diffeomorphisms $\left\{{g^t}_F \right\}$ satisfying:
\[ \left. \dfrac{d}{dt} \right|_{t=0} {g^t}_F x = I dF (x) \ . \]
If $F, G$ are two hamiltonian functions on $S^2$, then their \emph{Poisson bracket} is defined as: 
\[ \left\{ F , G \right\}(x) = \left. \dfrac{d}{dt} \right\vert_{t = 0} F({g^t}_G (x)) \ . \]
The Poisson bracket also satisfies the following formula: 
\[ \left\{ F, G \right\} = d F (I dG) = - \omega (I dF, I dG) \ . \]
For further reading on Poisson bracket we refer the reader to \cite{Ar1}.
\begin{rem} \label{rem:pb}
In this paper we are interested in the $L_1$-norm of the Poisson bracket. Note that on $S^2$ we have: 
\[ dF \wedge dG = - \left\{ F , G \right\} \cdot \omega \ , \]
therefore: 
\[\left\Vert \left\{ F, G \right\} \right\Vert_{L_1} = 
\int_{S^2} \left| \left\{ F , G \right\} \right| \omega = 
\int_{S^2} \left| dF \wedge dG \right| \ .\]
\end{rem}
\subsection{Zapolsky's inequality}
Zapolsky's inequality (\cite{Za}, theorem 1.4) relates the extent of non-linearity of a quasi-state to the $L_1$ norm of the Poisson bracket.
Let $\zeta$ be a representable quasi-state on $S^2$, then by Zapolsky's inequality for every $F,G \in C^1(S^2)$ we have: 
\[ \Pi(F,G)^2 \leq \Vert \left\{ F,G \right\} \Vert_{L_1} \ . \]
Note that this result can also be written as: 
\[ \sup_{F,G \in C^1 (S^2)} \dfrac{\Pi(F,G)^2}{\Vert \left\{ F,G \right\} \Vert_{L_1}} \leq 1 \ . \]
Our goal in this paper is to show that for some quasi-states Zapolsky's inequality is sharp. That is, we will show that there exist quasi-states for which: 
\[ \sup_{F,G \in C^{1} (S^2)} \dfrac{\Pi(F,G)^2}{\Vert \left\{ F, G \right\} \Vert_{L_1}} = 1\ . \]
\subsection{Main Results}
\paragraph*{}
\begin{thm} \label{thm:3pnt_max}
Let $\zeta$ be Aarnes' 3-point quasi-state, then:
\[ \max_{F,G \in C^{\infty} (S^2)} \dfrac{\Pi(F,G)^2}{\Vert \left\{ F,G \right\} \Vert_{L_1}} = 1 \ . \]
\end{thm}
\begin{thm} \label{thm:med_sup} 
Let $\omega$ be a normalized area form on $S^2$, that is $\int_{S^2} \omega = 1$, and $\zeta$ the corresponding median quasi-state. Then we have: 
\[ \sup_{F,G \in C^{\infty} (S^2) } \dfrac{ \Pi(F,G)^2}{\Vert \left\{ F,G \right\} \Vert_{L_1}}  = 1 \ . \]
\end{thm}
\section{Proofs}
\subsection{Proof of theorem \ref{thm:3pnt_max}}
\paragraph*{}
Prior to proving theorem \ref{thm:3pnt_max} we shall pay attention to the fact that any result we can prove for a certain 3-point quasi-state is true for all such quasi-states.
\begin{rem} \label{rem:3pnt_rat}
Let $\left\{ p_1,p_2,p_3 \right\}$ and $\left\{q_1,q_2,q_3 \right\}$ be two sets of three distinct points on the sphere $S^2$, and take $\zeta_1$ and $\zeta_2$ to be the two corresponding Aarnes' 3-point quasi-states and $\Pi_1$ and $\Pi_2$ the corresponding measurements of their non-linearity.
By a corollary to the isotopy lemma (see \cite{GP}, 3.6) there exists a diffeomorphism $h : S^2 \rightarrow S^2$ satisfying: 
\[h(p_i) = q_i \ , \ 1 \leq i \leq 3 \ . \]
Since $h$ is a diffeomorphism, both $h$ and $h^{-1}$ take solid subsets of the sphere to solid subsets, thus $\zeta_2(F \circ h) = \zeta_1(F)$ for every function $F \in C(S^2)$. Which yields: 
\[ \Pi_1 (F, G) = \Pi_2 (F \circ h, G \circ h) \ . \]
Also, we have: 
\begin{multline*}
\Vert \left\{ F \circ h, G \circ h \right\} \Vert_{L_1} = 
\int_{S^2} \vert d(F \circ h) \wedge d(G \circ h) \vert = 
\int_{S^2} \vert h^{\ast} \left( dF \wedge dG \right) \vert = \\
\int_{h(S^2)} \vert dF \wedge dG \vert = 
\int_{S^2} \vert dF \wedge dG \vert = 
\Vert \left\{ F , G \right\} \Vert_{L_1} \ . 
\end{multline*}
Thus: 
\[ \dfrac{\Pi_2 (F \circ h, G \circ h)^2}{\Vert \left\{ F \circ h , G \circ h \right\} \Vert_{L_1}} = \dfrac{\Pi_1 (F, G)^2}{\Vert \left\{ F,G \right\} \Vert_{L_1}} \ . \]
\end{rem}
\paragraph*{}
Based on this result we can prove the following theorem for a certain 3-point quasi-state and conclude that it is true for all such quasi-states.
\paragraph*{Proof of theorem \ref{thm:3pnt_max}}
\begin{proof}{}
Define 	$S^2 = \left\{ (x,y,z) \in \mathbb{R}^3 \ : \ x^2 + y^2 + z^2 = 1 \right\} \subset \mathbb{R}^3$. In spherical coordinates we have: \[S^2 = \left\{ (\cos{\theta} \cos{\phi} , \cos{\theta} \sin{\phi}, \sin{\theta} ) \ : \ -\frac{\pi}{2} \leq \theta \leq \frac{\pi}{2} \mbox{ and } 0 \leq \phi \leq 2 \pi \right\} \ . \]
Consider the following points on $S^2$: 
\[ \begin{array}{lcr} p_1 & = & (1,0,0) \\ 
					p_2 & = & (0,1,0) \\
					p_3 & = & (0,0,1) \end{array} \ . \]
Let $\zeta$ and $\tau$ be Aarnes' 3-point quasi-state and quasi-measure corresponding to these points.
\subparagraph*{}
Denote:
\[\begin{array}{l} 
\mathcal{D} = \left\{ (x,y) : x^2+y^2 \leq 1 \right\} \\
\Delta = \left\{ (u,v) \in \mathbb{R}^2 : u,v > 0 \mbox{ and } u+v < 1 \right\} \end{array} \ .\]
We build a continuous function $\widehat{\psi} : \mathcal{D} \rightarrow cl(\Delta)$ (see the figure below) satisfying : 
\begin{itemize}
\item $\widehat{\psi}$ maps the first quarter homeomorphically to $\Delta$ along the radii.
\item $\widehat{\psi}$ maps the second quarter to the segment $\left\{0\right\} \times [0,1]$ of the $y$-axis.
\item $\widehat{\psi}$ maps the third quarter to the origin $(0,0)$.
\item $\widehat{\psi}$ maps the fourth quarter to the segment $[0,1] \times \left\{0\right\}$ of the $x$-axis.
\end{itemize}
\begin{center}
\begin{tabular}{c}
\parbox[r]{10em}{
\begin{tikzpicture}
\draw[->] (-3,0) -- (3,0) node [label=right:$x$] {};
\draw[->] (0,-3) -- (0,3) node [label=above:$y$] {};

\draw [style=very thick] (0,0) circle (2);

\draw[->] (1.75, 0.8) -- (1.425, 0.625);
\draw[->] (1.35, 1.35) -- (1.05, 1.05);
\draw[->] (0.8, 1.75) -- (0.625, 1.425);

\draw[->] (-1.7, -0.5) -- (-0.4, -0.1);
\draw[->] (-1.25, -1.25) -- (-0.25, -0.25);
\draw[->] (-0.5, -1.7) -- (-0.1, -0.4);

\draw[->] (-1.8, 0.3) -- (-0.1, 0.3);
\draw[->] (-1.6, 0.7) -- (-0.1, 0.7);
\draw[->] (-1.3, 1.1) -- (-0.1, 1.1);
\draw[->] (-1.0, 1.5) -- (-0.1, 1.5);

\draw[->] (0.3, -1.8) -- (0.3, -0.1);
\draw[->] (0.7, -1.6) -- (0.7, -0.1);
\draw[->] (1.1, -1.3) -- (1.1, -0.1);
\draw[->] (1.5, -1.0) -- (1.5, -0.1);

\coordinate (O) at (0,0);
\coordinate [label=-45:$1$] (B) at (2,0);
\coordinate [label=135:$1$] (C) at (0,2);

\draw [style=very thick] (O) -- (B);
\draw [style=very thick] (O) -- (C);
\draw [style=very thick] (B) -- (C);

\draw (0.65,0.65) node {\begin{Large}$\Delta$\end{Large}};
\draw (2.3,2.3) node{I};
\draw (-2.3,2.3) node{II};
\draw (-2.3,-2.3) node{III};
\draw (2.3,-2.3) node{IV};

\end{tikzpicture}} \\
\begin{large}
$\qquad\qquad\widehat{\psi}(\mathcal{D})$
\end{large} 

\end{tabular}
\end{center}
We now build a smooth function $\psi : \mathcal{D} \rightarrow cl(\Delta)$ by smoothening $\widehat{\psi}$. 
For the precise definition of $\psi$ we will need an auxiliary smooth function $\alpha : \mathbb{R} \rightarrow [0,1]$ satisfying: 
\begin{center}
\begin{tabular}{lr}
\begin{tabular}{l}
$\alpha(s) = 0 \ , \ \forall s \leq 0$ \\
$\alpha(s) = 1 \ , \ \forall s \geq 1$ \\
$\alpha'(s) > 0 \ , \ \forall s \in (0,1)$.
\end{tabular} & 
\parbox[r]{10em}{
\begin{tikzpicture}
\draw[->] (-1,0) -- (2,0) node [label=right:$s$] {};
\draw[->] (0,-1) -- (0,2) node {};

\coordinate [label=225:$0$] (O) at (0,0);
\coordinate (A) at (-1,0);
\coordinate [label=below:$1$] (B) at (1,0);
\coordinate [label=left:$1$] (C) at (0,1);
\coordinate (D) at (1,1);

\draw [style=very thick] (A) -- (O);
\draw [style=very thick] (O) .. controls (0.1, 0) and (0.2, 0) .. (0.5, 0.8);
\draw [style=very thick] (0.5, 0.8) .. controls (0.55, 0.9) and (0.6, 1) .. (1,1);
\draw [style=very thick] (1,1) -- (2,1);
\draw [style=dashed] (B) -- (D);
\draw [style=dashed] (C) -- (D);

\draw (1.5,1.5) node {$\alpha(s)$};

\end{tikzpicture}}
\end{tabular}
\end{center}
Then we can define:
\[\rho(x,y) = \alpha \left( 2x^2+2y^2 - 1 \right) \cdot \alpha \left( \dfrac{x+y}{\sqrt{x^2+y^2}} \right) \ , \ \forall (x,y) \in \mathcal{D} \setminus \left\{ (0,0) \right\} \ ,\]
the images below illustrate the behaviour of this function. 
\begin{center}
\begin{tabular}{ccc}
\includegraphics[scale=0.24]{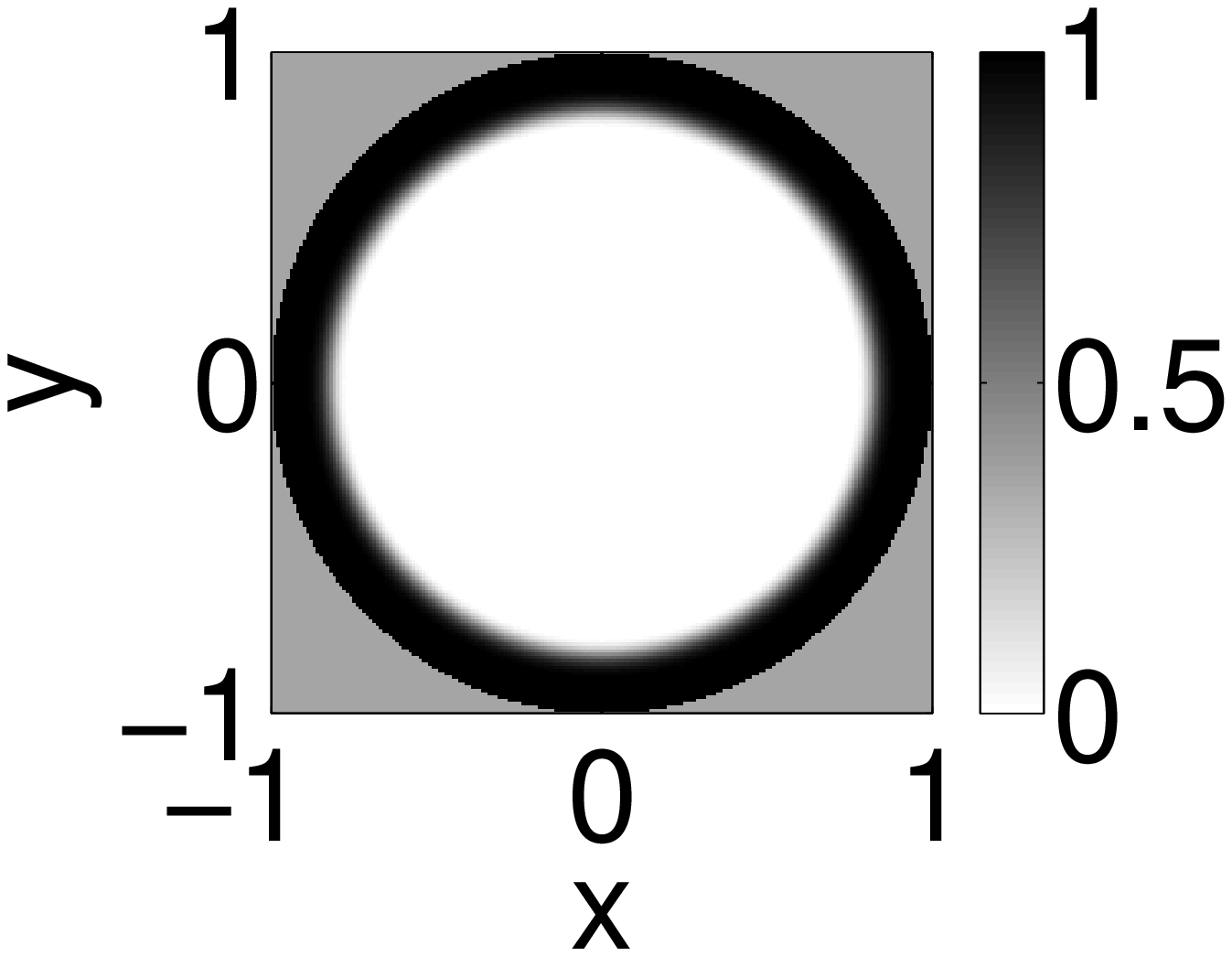} &
\includegraphics[scale=0.24]{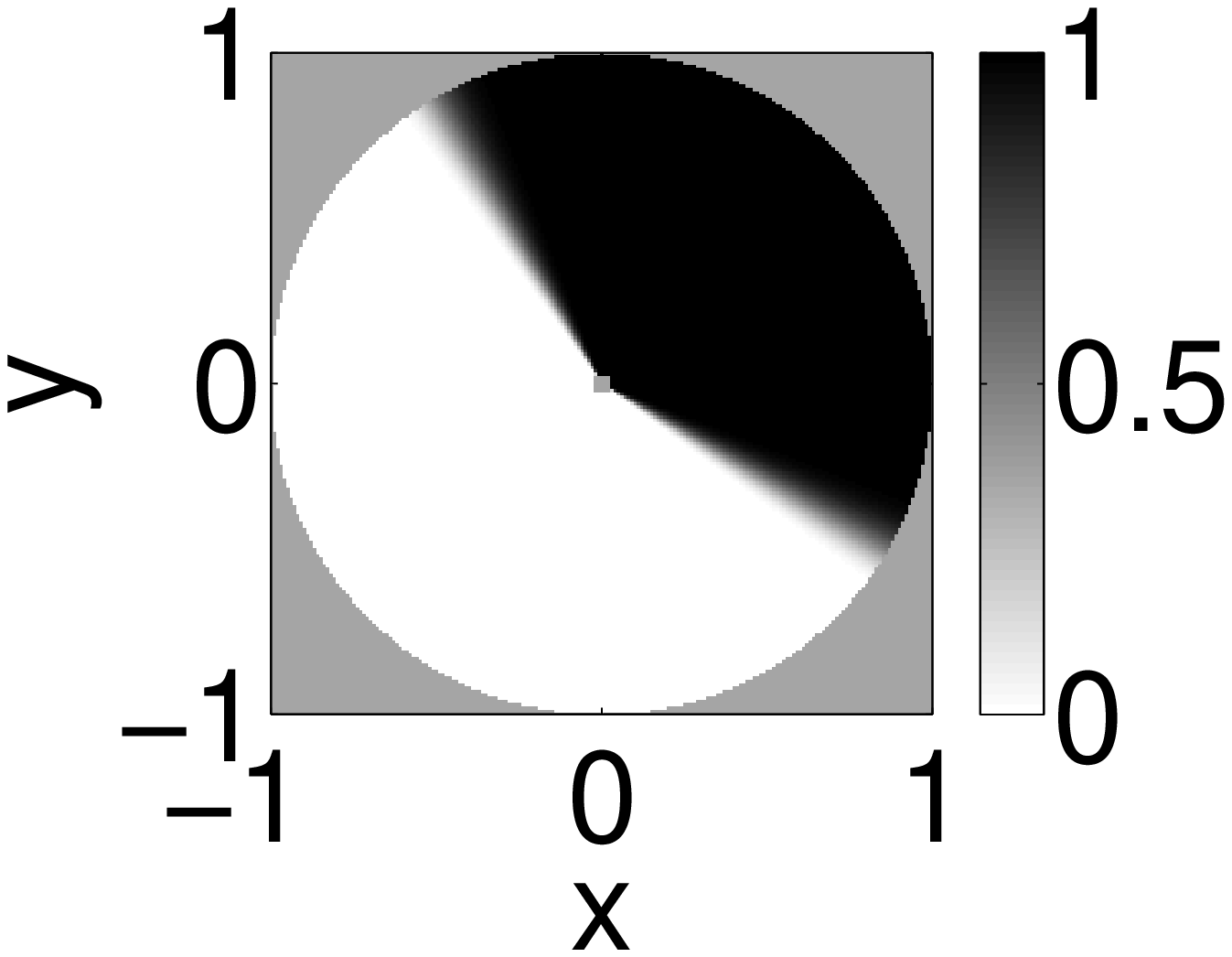} & 
\includegraphics[scale=0.24]{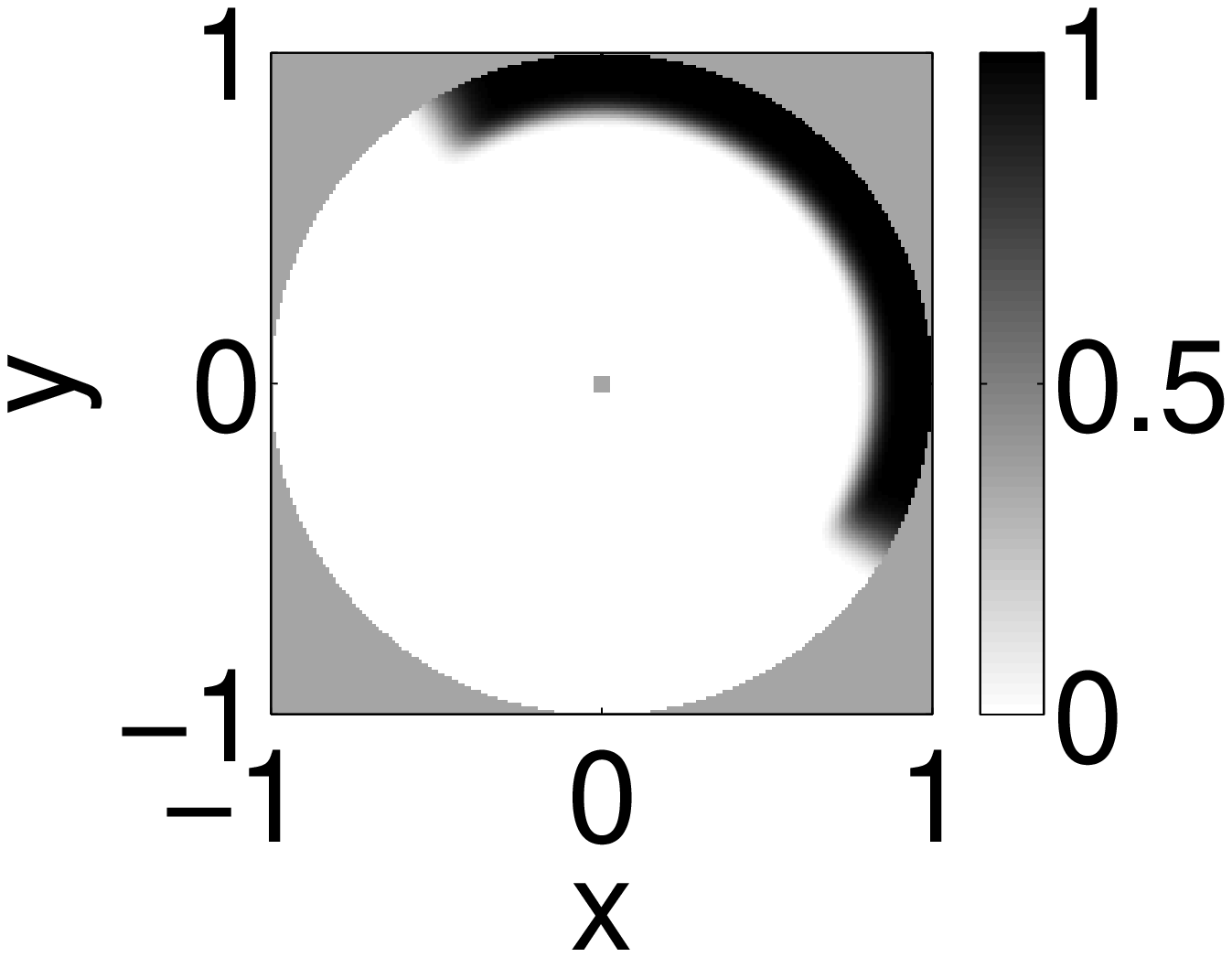}\\
$\alpha(2x^2+2y^2-1)$ & 
$\alpha\left(\frac{x+y}{\sqrt{x^2+y^2}}\right)$ & 
$\rho(x,y)$ \\
\end{tabular}
\end{center}
And we take $\psi : \mathcal{D} \rightarrow cl(\Delta)$ to be $\psi(p) = (f(p), g(p))$, whereas:  
\[ \begin{array}{lcr} 	
f(x,y) &=& 
\left\{ \begin{array} {lcl}
x \leq 0 & , & 0 \\
0 < x & , & 
\rho(x,y) \cdot 
\dfrac{ \alpha\left(\frac{x}{\sqrt{x^2+y^2}}\right)}{ \alpha\left(\frac{x}{\sqrt{x^2+y^2}}\right) + \alpha\left(\frac{y}{\sqrt{x^2+y^2}}\right)} \end{array} \right. \\
g(x,y) &=& 
\left\{ \begin{array} {lcl}
y \leq 0 & , & 0 \\
0 < y & , & 
\rho(x,y) \cdot 
\dfrac{ \alpha\left(\frac{y}{\sqrt{x^2+y^2}}\right)}{ \alpha\left(\frac{x}{\sqrt{x^2+y^2}}\right) + \alpha\left(\frac{y}{\sqrt{x^2+y^2}}\right)} \end{array} \right.
\end{array} \ .\]
\begin{lem} \label{lem:smooth} 
$f,g$ are smooth functions.
\end{lem}
\begin{proof}
The proofs of smoothness for $f$ and $g$ are very similar, therefore we will give the proof only for $f$. To show that $f$ is smooth, we need to show that it is smooth on every point of its domain. Take a point $(x_0, y_0) \in \mathcal{D}$,  and consider the following cases: 
\begin{itemize}
\item If $x_0 < 0$, then $f$ is identically zero in a neighbourhood of $x_0$, and hence smooth.
\item If $x_0 > 0$, then $x>0$ in a neighbourhood of $x_0$, hence $\sqrt{x^2+y^2} > 0$ and $f$ is a multiplication of smooth functions divided by smooth $\alpha\left(\frac{x}{\sqrt{x^2+y^2}}\right) + \alpha\left(\frac{y}{\sqrt{x^2+y^2}}\right)$. But, since $x > 0$, $\alpha\left(\frac{x}{\sqrt{x^2+y^2}}\right) > 0$, therefore the denominator $\alpha\left(\frac{x}{\sqrt{x^2+y^2}}\right) + \alpha\left(\frac{y}{\sqrt{x^2+y^2}}\right) > 0$, and $f$ is smooth.
\item If $x_0 = 0$ and $y_0 < 0$, we can find a neighbourhood $U$ of $(x_0, y_0)$ on which $y < 0$ and $x+y \leq 0$. But then in this neighbourhood we have $x^2+y^2>0$ and $\frac{x+y}{\sqrt{x^2+y^2}} \leq 0$, hence $\alpha\left(\frac{x+y}{\sqrt{x^2+y^2}}\right) = 0 \Rightarrow \rho(x,y) = 0$, which yields:
\[ f(x,y) \vert_U = \left\{ \begin{array}{lcl} 
x > 0 & , & 0 \cdot \dfrac{\alpha\left(\frac{x}{\sqrt{x^2+y^2}}\right)}{\alpha\left(\frac{x}{\sqrt{x^2+y^2}}\right) + \alpha\left(\frac{y}{\sqrt{x^2+y^2}}\right)} = 0 \\
x \leq 0 & , & 0 \end{array} \right. \ . \]
Thus $f$ is identically zero in this neighbourhood, and hence smooth.
\item If $x_0 = 0$ and $y_0 > 0$, we can find a neighbourhood of $y_0$ such that $y>0$, thus $\alpha\left(\frac{y}{\sqrt{x^2+y^2}}\right) > 0 $. In this neighbourhood the denominator, $\alpha\left(\frac{x}{\sqrt{x^2+y^2}}\right) + \alpha\left(\frac{y}{\sqrt{x^2+y^2}}\right) > 0$, thus $f$ will be the multiplication of smooth functions for $x > 0$ and zero for $x \leq 0$. From $\alpha$'s smoothness we have $\lim_{s \rightarrow 0} \alpha^{(m)} (s) = 0$ for every derivative $m \in \mathbb{N}$. If $x>0$, every derivative of $f$ will be a finite sum of products, each of which has a multiplicand of the form $\alpha^{(m)}\left(\frac{x}{\sqrt{x^2+y^2}}\right)$ for some $m \in \mathbb{N}$, therefore: 
\[\lim_{(x,y) \rightarrow (0, y_0)} f^{(n)} (x,y) = 0 \ , \ \forall n \in \mathbb{N} \ , \]
and $f$ is smooth.
\item Finally, if $x_0 = 0$ and $y_0 = 0$, we can find a neighbourhood of $(x_0,y_0)$ on which we have $x^2+y^2< \frac{1}{2}$. But then: 
$\alpha(2x^2 + 2y^2 - 1) = 0 \Rightarrow \rho(x,y) = 0$, 
and hence $f$ is identically zero in this neighbourhood, thus smooth.
\end{itemize}
We have shown that $f$ is smooth on every point of $\mathcal{D}$, thus $f$ is a smooth function. In a similar manner it can be shown that $g$ is also smooth. 
\end{proof}
Denote:
\[A = \left\{ (x,y) \in \mathcal{D} : x,y > 0 \mbox{ and } \frac{1}{2} < x^2+y^2 < 1 \right\} \ .\]
\begin{lem} \label{lem:bijection}
The restriction $\psi \vert_{A}$ is one-to-one and onto $\Delta$.
Also, $\psi(\mathcal{D} \setminus A) \subset \partial \Delta$.
\end{lem}
\begin{proof}
On $A$ we have $x,y>0$, thus $x+y > \sqrt{x^2+y^2}$ and $\alpha\left(\frac{x+y}{\sqrt{x^2+y^2}}\right) = 1$. Hence: \[\left. \left( f+g \right) \right|_A = \rho \vert_{A} = \alpha(2x^2+2y^2-1) \ .\]
Similarly:
\[\left.\left( \frac{f}{g} \right) \right|_A = \dfrac{\alpha\left(\frac{x}{\sqrt{x^2+y^2}}\right)}{\alpha\left(\frac{y}{\sqrt{x^2+y^2}}\right)} \ . \]
In spherical coordinates we have: 
\[A = \left\{(\cos{\theta}\cos{\phi}, \cos{\theta}\sin{\phi}) : 0 < \phi < \frac{\pi}{2} \mbox{ and } 0 < \theta < \frac{\pi}{4} \right\} \ . \]
Therefore: 
\[\left.\left(f+g\right) \right|_A = \alpha(2cos^2{\theta}-1) \mbox{ and } 
\left.\left( \frac{f}{g} \right) \right|_A = \dfrac{\alpha(\cos{\phi})}{\alpha(\sin{\phi})} \ . \]
Note that $\alpha$ is a bijection of $(0,1)$ to $(0,1)$ and $\left(2\cos^2{\theta}-1\right)$ is a bijection of $\left(0, \frac{\pi}{4}\right)$ to $(0,1)$, hence $\left( f+g \right) \vert_A = \alpha(\cos^2{\theta}-1)$ is a bijection of $\left(0, \frac{\pi}{4}\right)$ to $(0,1)$.
Also: 
\begin{multline*}
\dfrac{d \left. \left( \frac{f}{g} \right) \right|_{A}}{d \phi} = 
\dfrac{ d \left( \frac{\alpha(\cos{\phi})}{\alpha(\sin{\phi})} \right)}{d \phi} = \\
-\dfrac{\alpha'(\cos{\phi}) \cdot \alpha(\sin{\phi}) \cdot sin{\phi} + \alpha(\cos{\phi}) \cdot \alpha'(\sin{\phi}) \cdot \cos{\phi}}{\alpha^2(\sin{\phi})} 
\end{multline*}
Recall that $\alpha(s), \alpha'(s) > 0$ for $s \in (0,1)$ and that on $A$ we have $0< \cos{\phi}, \sin{\phi} < 1$, therefore: $\dfrac{d \left. \left( \frac{f}{g} \right) \right|_{A}}{d \phi} < 0$, and $\left. \left( \frac{f}{g} \right) \right|_A$ is a bijection of $\left(0, \frac{\pi}{2} \right)$ to $(0, \infty)$.
\subparagraph*{}
We have shown that $\left.\left(f+g, \frac{f}{g}\right)\right|_A$ is a bijection of $A$ to $(0,1) \times (0,\infty)$. Since $\left(u+v, \frac{u}{v}\right)$ is a bijection of $\Delta$ to $(0,1) \times (0,\infty)$, $\psi \vert_A$ is a bijection of $A$ to $\Delta$.
\subparagraph*{}
We still have to show that $\psi(\mathcal{D} \setminus A) \subset \partial \Delta$.
Note that a point $(x,y) \in \mathcal{D} \setminus A$ satisfies at-least one of these four conditions:
\begin{itemize}
\item $x \leq 0$ \\
In this case we have $f(x,y) = 0$ and $\psi(x,y) = (0, g(x,y)) \in \partial \Delta$.
\item $y \leq 0$ \\
Similarly $g(x,y) = 0$ and $\psi(x,y) = (f(x,y), 0) \in \partial \Delta$.
\item $x^2+y^2 \leq \frac{1}{2}$ \\
Here $\rho(x,y) = \alpha(2x^2+2y^2-1) \cdot \alpha(\frac{x+y}{\sqrt{x^2+y^2}}) = 0$ and $\psi(x,y) = (0,0) \in \partial \Delta$.
\item $x,y > 0$ and $x^2+y^2 = 1$ \\
Here $\rho(x,y) = \alpha(2x^2+2y^2-1) \cdot \alpha\left(\frac{x+y}{\sqrt{x^2+y^2}}\right) = 1$. hence $f(x,y) + g(x,y) = 1$ and $\psi(x,y) = (f(x,y), g(x,y)) \in \partial \Delta$.
\end{itemize}
Thus we have shown that $\psi(\mathcal{D} \setminus A) \subset \partial \Delta$.
\end{proof}
Let $P:S^2 \rightarrow \mathbb{R}^2$ be the projection of the sphere to the $xy$-plane. Define: $F,G:S^2 \rightarrow \mathbb{R}$ by $F = f \circ P$ and $G = g \circ P$. 
Our goal is to show that: 
\[ \Pi^2(F,G) = \Vert \left\{ F, G \right\} \Vert_{L_1} \ . \]
We will begin by proving the following lemma:
\begin{lem} \label{lem:qs1}
\[\Pi(F,G) = 1 \ . \]
\end{lem}
\begin{proof}
Note: 
\[ \begin{array}{lcr} 	(F,G)(p_1) &=& (1,0) \\
						(F,G)(p_2) &=& (0,1) \\
						(F,G)(p_3) &=& (0,0) \end{array} . \]
\subparagraph*{}
Since $p_2, p_3 \in \left\{ (x,y,z) \in S^2 : x \leq 0 \right\}$, and since the half-sphere is a solid subset of the sphere we have $\tau(\left\{ (x,y,z) \in S^2 : x \leq 0 \right\}) = 1$. Also: 
\[\left\{ (x,y,z) \in S^2 : x \leq 0 \right\} \subset F^{-1} (0) \subset \left\{ F < t \right\} \ , \ \forall t > 0 \ . \]
Therefore: 
\[b_F (t) = \tau (\left\{ F < t \right\} ) = 1 \ , \ \forall t > 0 \ . \]
\subparagraph*{}
In the same way we have $p_1, p_3 \in \left\{ (x,y,z) \in S^2 : y \leq 0 \right\}$, and as this half-sphere is also a solid subset, we get once more $\tau(\left\{ (x,y,z) \in S^2 : y \leq 0 \right\}) = 1$. As before: \[ \left\{ (x,y,z) \in S^2 : y \leq 0 \right\} \subset G^{-1} (0) \subset \left\{ G < t \right\} \ , \ \forall t > 0 \ . \]
Thus: 
\[b_G (t) = \tau (\left\{ G < t \right\} ) = 1 \ , \ \forall t > 0 \ . \]
\subparagraph*{}
Last it should be noted that the arc: 
\[\left\{ (x,y,0) \in S^2 : x,y \geq 0 \text{ and } x^2+y^2 = 1 \right\} \]
is also a solid subset of the sphere, and that:
\[p_1, p_2 \in \left\{ (x,y,0) \in S^2 : x,y \geq 0 \text{ and } x^2 + y^2 = 1 \right\} \ . \]
Therefore: 
\[\tau(\left\{ (x,y,0) \in S^2 : x,y \geq 0 \text{ and } x^2 + y^2 = 1 \right\}) = 1 \ . \]
Since: 
\[\left\{ (x,y,0) \in S^2 : x,y \geq 0 \text{ and } x^2 + y^2 = 1 \right\} \subset (F+G)^{-1} (1) \ , \]
we have: \[\tau((F+G)^{-1}(1)) = 1 \ . \]
Therefore, the quasi-measure of its complement $\left\{ F+G < 1 \right\}$ is $0$. For every $0 \leq t \leq 1$, $\left\{ F+G < t \right\}$ is a subset of $\left\{ F+G < 1 \right\}$, thus:
\[b_{F+G} (t) = \tau( \left\{ F+G < t \right\}) = 0 , \forall t \leq 1 \ . \]
\subparagraph*{}
Hence: 
\begin{eqnarray*}
\zeta(F) &=& 1 - \int_{0}^{1} b_F (t) dt = 1 - \int_{0}^{1} 1 dt = 0 \\
\zeta(G) &=& 1 - \int_{0}^{1} b_G (t) dt = 1 - \int_{0}^{1} 1 dt = 0 \\
\zeta(F+G) &=& 1 - \int_{0}^{1} b_{F+G} (t) dt = 1 - \int_{0}^{1} 0 dt = 1 \ .
\end{eqnarray*}
And we get that: 
\[\Pi(F,G) = \vert \zeta(F+G) - \zeta(F) - \zeta(G) \vert = \vert 1 - 0 - 0 \vert = 1 \ . \]
\end{proof}
We now have to compute $\Vert \left\{ F,G \right\} \Vert_{L_1}$. 
Recall that: 
\[\Vert \left\{ F,G \right\} \Vert_{L_1} = \int_{S^2} \vert dF \wedge dG \vert = \int_{S^2} \vert (\psi \circ P)^{\ast} (dx \wedge dy) \vert \ .\]
From lemma \ref{lem:smooth}, $\psi \circ P$ is a smooth function, then, as a corollary to the change of variables formula for a many-to-one function (see \cite{Ta}, theorem F.1) we have: 
\[\int_{S^2} \vert (\psi \circ P)^{\ast} (dx \wedge dy) \vert =  \int_{\psi \circ P (S^2)} n(x,y) \cdot dx \wedge dy \ , \]
with: \[n(x,y) = card((\psi \circ P)^{-1} (x,y)) \ . \]
Also, by lemma \ref{lem:bijection}, we know that $\psi \circ P$ covers $\Delta$ exactly twice (since $P$ projects the sphere twice onto $A$), hence $n(x,y) = 2$ for $(x,y) \in \Delta$. Thus:
\[\int_{S^2} \vert (\psi \circ P)^{\ast} (dx \wedge dy) \vert =  \int_{cl(\Delta)} n(x,y) dx \wedge dy = 
\int_{\Delta} 2 dx \wedge dy = 1 \ . \]
\subparagraph*{}
Thus we have shown that for Aarnes' 3-point quasi-state corresponding to these specific three points $p_1, p_2, p_3$ we have: 
\[ \dfrac{ \Pi(F,G)^2 }{ \Vert \left\{ F,G \right\} \Vert_{L_1} } = \dfrac{1^2}{1} = 1 \ . \]
Remark \ref{rem:3pnt_rat} concludes this proof for any 3-point quasi-state.
\end{proof}
\subsection{Proof of theorem \ref{thm:med_sup}}
\paragraph*{}
In the proof of theorem \ref{thm:med_sup} we will use the fact that diffeomorphisms preserve the relation between the extent of non-linearity of  median quasi-states and the $L_1$-norm of the Poisson bracket. 
\begin{rem} \label{rem:med_rat}
Let $h : M_1 \rightarrow M_2$ be a diffeomorphism of surfaces. If $\omega$ is an area form on $M_2$ then $h^{\ast} \omega$ is an area form on $M_1$. Take $\zeta_1$ and $\zeta_2$ to be the median quasi-states corresponding to $h^{\ast} \omega$ and $\omega$.
Recall that $m_F$, the median of a function $F \in C^{1} (M_2)$, is the unique connected component of the level set $F^{-1} (\zeta_2(F)) \subset M_2$ satisfying $\int_{B} \omega \leq \frac{1}{2} \int_{M_2} \omega$ for each connected component $B$ of $M_2 \setminus m_F$.
Since $h,h^{-1}$ are continuous functions, they take connected sets to connected sets, therefore $h^{-1} (m_F)$ is a connected component of the level set $(F \circ h)^{-1} (\zeta_2(F)) \subset M_1$. If $A$ is a connected component of $M_1 \setminus h^{-1} (m_F)$, then $h(A)$ must be a connected component of $M_2 \setminus (m_F)$. Therefore: 
\[\int_{A} h^{\ast} \omega = \int_{h(A)} \omega \leq \frac{1}{2} \int_{M_2} \omega = \frac{1}{2} \int_{M_1} h^{\ast} \omega \ . \]
Thus $h^{-1} (m_F)$ must be the median of the function $F \circ h$, which yields:  \[ \zeta_1 (F \circ h) = \zeta_2(F) \ . \]
Therefore if $\Pi_1$ and $\Pi_2$ are the extents of non-linearity of the quasi-states $\zeta_1$ and $\zeta_2$, we get:
\[ \Pi_1 (F \circ h, G \circ h) = \Pi_2 (F , G) \ . \]
Also, we have: 
\begin{multline*}
\Vert \left\{ F \circ h, G \circ h \right\} \Vert_{L_1} = 
\int_{M_1} \vert d(F \circ h) \wedge d(G \circ h) \vert = 
\int_{M_1} \vert h^{\ast} \left( dF \wedge dG \right) \vert = \\
\int_{h(M_1)} \vert dF \wedge dG \vert = 
\int_{M_2} \vert dF \wedge dG \vert = 
\Vert \left\{ F \circ G \right\} \Vert_{L_1} \ . 
\end{multline*}
Thus: 
\[ \dfrac{ \Pi_1 (F \circ h, G \circ h)^2 }{\Vert \left\{ F \circ h, G \circ h \right\} \Vert_{L_1}} = \dfrac{ \Pi_2 (F, G)^2}{\Vert \left\{ F, G \right\} \Vert_{L_1}} \ , \]
and: 
\[ \sup_{F,G \in C^{\infty} (M_1)} \dfrac{\Pi_1 (F, G)^2}{\Vert \left\{ F,G \right\} \Vert_{L_1}} = \sup_{F,G \in C^{\infty} (M_2)} \dfrac{ \Pi_2 (F,G)^2}{\Vert \left\{ F,G \right\} \Vert_{L_1}} \ . \]
\end{rem}
\newpage
\paragraph*{Proof of theorem \ref{thm:med_sup}}
\begin{proof}
Consider the triangle $ABC$ with vertices: 
\[\left\{ \begin{array}{lcr} A&=&(0,0) \\ B&=&(0,1) \\ C&=&(1,0) \end{array} \right. \ \]
in the $xy$-plane.
For $ \frac{1}{4} > \epsilon > 0$ draw the segments $DK$, $EJ$, $IL$ with:
\[ \left\{ \begin{array}{lcr} D&=&(0, \epsilon) \\ E&=&(0, 1-\epsilon) \\ I&=&(\epsilon, 0) \end{array} \qquad \begin{array}{lcr} K&=&(1-\epsilon, \epsilon) \\ J&=&(1-\epsilon, 0)  \\ L&=&(\epsilon, 1-\epsilon) \end{array} \right. \ . \]
Let $U$ be the triangle $\triangle ABC$ after smoothing its corners by curves that do not intersect the segments $DK$, $EJ$ and $IL$. Then the segments $DK$, $EJ$ and $IL$ divide $U$ into seven parts, $U_1, U_2, \ldots, U_7$.
\begin{center}

\begin{tikzpicture}
\draw[->,very thick] (-1,0) -- (6.5,0) node [label=right:$x$] {};
\draw[->,very thick] (0,-1) -- (0,6.5) node [label=above:$y$] {};

\coordinate [label=225:$A$] (A) at (0,0);
\coordinate [label=right:$B$] (B) at (0,6);
\coordinate [label=above:$C$] (C) at (6,0);
\coordinate [label=45:$D$] (D) at (0,1);
\coordinate [label=135:$E$] (E) at (0,5);
\coordinate [label=45:$I$] (I) at (1,0);
\coordinate [label=45:$J$] (J) at (5, 0);
\coordinate [label=right:$K$] (K) at (5,1);
\coordinate [label=right:$L$] (L) at (1,5);

\draw (A) -- (B) -- (C) -- (A);
\draw[rounded corners=16pt] (A) -- (B) -- (C) -- (A) -- (B);
\draw (D) -- (K);
\draw (E) -- (J);
\draw (I) -- (L);

\draw (-0.25, 6) node {$1$};
\draw (-0.25, 1) node {$\epsilon$};
\draw (-0.5, 5) node {$1-\epsilon$};
\draw (6, -0.25) node {$1$};
\draw (1, -0.25) node {$\epsilon$};
\draw (5, -0.25) node {$1-\epsilon$};

\draw (0.5,0.5) node {$U_1$};
\draw (0.5, 5) node {$U_2$};
\draw (5, 0.5) node {$U_3$};
\draw (0.5, 2) node {$U_4$};
\draw (3.5, 2) node {$U_5$};
\draw (2, 0.5) node {$U_6$};
\draw (2,2) node {$U_7$};
\end{tikzpicture}

\end{center}
\subparagraph*{}
Note that $U_7 \subset U \subset \triangle ABC$, and hence: 
\begin{equation} \label{eq:area}
\frac{(1-3 \epsilon)^2}{2} = Area(U_7) < Area(U) < Area(\triangle ABC) = \frac{1}{2} \ . 
\end{equation}
\subparagraph*{}
Let $u : U \rightarrow [0, \infty)$ be a function satisfying $u^{-1} (0) = \partial U$ with $0$ a regular value of $u$. And take $S$ to be the surface in $\mathbb{R}^3$ defined as $S := \left\{ z^2 = u(x,y) \right\}$.\\
Consider the following functions: 
\begin{itemize}
\item $P : S \rightarrow \mathbb{R}^2$ defined as $P(x,y,z) = (x,y)$ is the projection of $S$ to the plane. Note that $S \setminus P^{-1}(\partial U)$ has two connected components, 
\[ \left\{ (x,y, \pm \sqrt{ u(x,y) }) : (x,y) \in int(U) \right\} \ , \]
both of which are projected diffeomorphically to $int(U)$ by $P$.
\item $F : S \rightarrow \mathbb{R}$ defined as $F(x,y,z) = x$.
\item $G : S \rightarrow \mathbb{R}$ defined as $G(x,y,z) = y$.
\end{itemize}
Then by (\ref{eq:area}) we get: 
\begin{equation*}
\Vert \left\{ F,G \right\} \Vert_{L_1} = \int_S \vert dF \wedge dG \vert = 
\int_S \vert dx \wedge dy \vert = 2 \cdot Area (U) \in ( (1-3 \epsilon)^2, 1) \ .
\end{equation*}
\subparagraph*{}
Let $\sigma$ be an area form on $S$ such that:
\[ \int_{P^{-1} (U_1)} \sigma = \int_{P^{-1} (U_2)} \sigma = \int_{P^{-1} (U_3)} \sigma = \frac{2}{10} \]
and
\[ \int_{P^{-1} (U_4)} \sigma = \int_{P^{-1} (U_5)} \sigma = \int_{P^{-1} (U_6)} \sigma = \int_{P^{-1} (U_7)} \sigma = \frac{1}{10} \ . \]
Note that $\sigma$ is a normalized area form on $S$, and that each of the curves $P^{-1} (IL), P^{-1} (DK)$ and $P^{-1} (EJ)$ divides $S$ into two disks, one of area: 
\[ \frac{2}{10} + \frac{1}{10} + \frac{2}{10} = \frac{5}{10} = \frac{1}{2} \]
and the second of area: 
\[ \frac{1}{10} + \frac{1}{10} + \frac{1}{10} + \frac{2}{10} = \frac{5}{10} = \frac{1}{2} \ . \]
Then, if $\zeta$ is the median quasi-state corresponding to $\sigma$, we get: 
\[ \left\{ \begin{array}{lclcr} 
\zeta(F) & = & F(IL) & = & \epsilon \\
\zeta(G) & = & G(DK) & = & \epsilon \\
\zeta(F+G) & = & (F+G) (EJ) & = & 1-\epsilon \end{array} \right. .\]
\subparagraph*{}
Therefore: 
\[ \dfrac{ \Pi (F,G)^2}{\Vert \left\{ F,G \right\} \Vert_{L_1}} \geq \dfrac{ \vert 1-\epsilon - \epsilon - \epsilon \vert^2}{1}  \xrightarrow[\epsilon \rightarrow 0]{} 1 \ , \]
and hence we have: 
\[ \sup_{F,G \in C^{\infty}(S)} \dfrac{ \Pi (F,G)^2 }{\Vert \left\{ F,G \right\} \Vert_{L_1}} = 1 \ .\]
\subparagraph*{}
Note that $U$ is diffeomorphic to a closed disk, hence $S$ is diffeomorphic to the sphere, and there exists a diffeomorphism $h_1 : S^2 \rightarrow S$. Recall that $\sigma$ is a normalized area form on $S$, hence $\sigma_1 = {h_1}^{\ast} \sigma$ is a normalized area form on $S^2$. Let $\Pi_1$ be the extent of non-linearity of the median quasi-state corresponding to $\sigma_1$, then by remark \ref{rem:med_rat} we have: 
\[ \sup_{F,G \in C^{\infty} (S^2)} \dfrac{\Pi_1(F,G)^2}{\Vert \left\{ F,G \right\} \Vert_{L_1}} = \sup_{F,G \in C^{\infty}(S)} \dfrac{ \Pi (F,G)^2}{\Vert \left\{ F,G \right\} \Vert_{L_1}} = 1 \ . \]
If $\sigma_2$ is another normalized area form on $S^2$, then by Moser's theorem (see \cite{La}, 13.2) there exists a diffeomorphism $h_2 : S^2 \rightarrow S^2$, such that $\sigma_2 = {h_2}^{\ast} \sigma_1$. If $\Pi_2$ is the extent of non-linearity of the median quasi-state corresponding to $\sigma_2$, then by using remark \ref{rem:med_rat} again, we will get: 
\[ \sup_{F,G \in C^{\infty} (S^2)} \dfrac{ \Pi_2(F,G)^2}{\Vert \left\{ F,G \right\} \Vert_{L_1}} = \sup_{F,G \in C^{\infty}(S^2)} \dfrac{\Pi_1 (F,G)^2}{\Vert \left\{ F,G \right\} \Vert_{L_1}} = 1 \ . \]
Thus, for every normalized area form $\omega$ on $S^2$, if $\Pi$ is the extent of non-linearity of its corresponding median quasi-state, we have: 
\[ \sup_{F,G \in C^{\infty} (S^2)} \dfrac{ \Pi(F,G)^2}{\Vert \left\{ F,G \right\} \Vert_{L_1}} = 1 \ . \]
\end{proof}

\section*{Acknowledgements}
This work is an M.Sc. thesis prepared under the supervision of Prof. Leonid Polterovich, the department of pure mathematics, Tel-Aviv university.\\
I would like to thank both my supervisor Prof. Leonid Polterovich, and my colleague Dr. Frol Zapolsky, for their help and support.\\
This research was partially supported by the Israel Science Foundation grant 509/07.

\end{document}